%% file: for_arvix.tex
\documentclass{article}
\usepackage{geometry}                
\usepackage[pdftex]{graphicx}

\usepackage{amssymb}
\usepackage{epstopdf}
\usepackage{algorithmic}
\usepackage{algorithm}
\usepackage{color}
\usepackage{subfig}
\usepackage{authblk}

\input{defs}

\title{Novel Method for Background Phase Removal on MRI Proton Resonance Frequency Measurements}

\author[1,3]{Wolfgang~Stefan}
\author[1]{David~Fuentes}
\author[2]{Erol~Yeniaras}
\author[1]{Ken-Pin~Hwang}
\author[1]{John~D.~Hazle}
\author[1]{R.~Jason~Stafford}
\affil[1]{Departments of Imaging Physics, The University of Texas MD Anderson Cancer Center, Houston, Texas, 77054}
\affil[2]{erol.yeniaras@gmail.com}
\affil[3]{Corresponding author: wstefan@mdanderson.org}

\begin{document}
\maketitle

\section{Introduction}
%
%
%
%
Phase data provides information for  a variety
of applications in magnetic resonance imaging (MRI) including
shimming, artifact removal, temperature imaging, and improving soft tissue contrast  
\cite{Schweser2010,fuentesetal12a,Wharton2010,Rauscher:2005,Duyn:2007}. 
One of the most prominent examples of using phase images as an
additional contrast mechanism is susceptibility weighted imaging (SWI), which
is useful in a wide variety of neuro-pathologies, e.g., traumatic brain injury,
cerebral microhemorrhages, cerebrovascular disease, multiple sclerosis,
intracranial hemorrhage, and tumors \cite{Haacke2004,Haacke_SWI_book}.
Quantitative susceptibility mapping is another useful application of using
phase images to complement magnitude data \cite{Rochefort2008,Zhou2014}.
The MR phase data is directly related to the proton
resonance frequency (PRF) when gradient echo imaging sequences are used. In
order to make best use of phase data it is necessary to accurately separate the
background phase (e.g., generated by local field inhomogeneities, air-tissue
interfaces, etc.) from the phase contributions of tissue magnetic susceptibility
differences. 

Reference scans, in which an identical object is scanned with the
susceptibility sources removed, provides the most reliable background field
measurements \cite{Rochefort2008,Liu2009}.

However, it is impractical or impossible to perform a reference scan in many
{\it in vivo} situations.

The background can be estimated and removed, if {\it a priori} knowledge of the
spatial distribution of all background susceptibility sources
\cite{Neelavalli2009,Koch2006} can be assumed. The assumption of the
separability between local and background fields in a certain space is often
violated, leading to erroneous estimation of local fields, and the results
depend on the choice of the basis functions \cite{Langham2009}. In practice,
the knowledge of the background susceptibility source surrounding a given
region of interest (ROI) is often not fully available or sufficiently accurate.
This leads to substantial residual background field that requires additional
attention, particularly when there are significant variations in susceptibility
outside the imaging field of view (FOV) \cite{Neelavalli2009}. 

Current background removal techniques, that do not require prior knowledge of
the background susceptibility source assume that the local and background
fields are in a space spanned by the Fourier basis \cite{Wang2000} or
polynomial functions are separable \cite{Rochefort2008,Langham2009,Yao2009}.
Usually, a user defined (or automatically determined) ROI has to be defined. In
\cite{Grissom2010} the selection of the ROI for the polynomial interpolation
problem is avoided using a reweighed $\ell 1$ regression. 

Another method to estimate the background is to solve a Dirchlet boundary
problem of the Laplace or Poisson equation \cite{Solomir2012,Zhou2014,Li:2014}. It is argued in
\cite{Solomir2012} that the Laplace solution is potentially more stable than
polynomial interpolation and provides physically more meaningful solutions
because the magnetic field and thus the PRF in homogeneous tissue has the
property $\vec{\nabla}^2 \Phi = 0$. However this method also relies on the
manual selection of an ROI.

In this paper, we seek to remove the background without the use of a reference
PRF map, an ROI or an heuristic basis like the Fourier or polynomials. It can
be interpreted as physically meaningful similar to the arguments in
\cite{Solomir2012,Zhou2014}, and has very limited computational costs. On the basis of
its smoothness, the proposed method suppresses the background, rather than
estimating and subtracting it. We assume that sharp changes in the image are
either due to tissue-tissue interfaces or due to the heating, while the
background is smooth (i.e., does not contain very large gradients). We also
assume that sharp gradients are sparse, i.e., only a few pixels contain either
tissue-tissue interfaces or heating, and that the PRF for the rest of the image
is a harmonic function. Our background suppression algorithm has two stages;
First sharp gradients are detected in the phase image, using a sparsity
enforcing edge detection method. Second a corrected PRF map is found, which has
the same sharp gradients, and satisfies $\vec{\nabla}^2 \Phi = 0$ in the smooth
regions.  As a consequence the difference between the original PRF map and the
corrected PRF map is a harmonic function, which is physically meaningful as
argued in \cite{Solomir2012,Zhou2014}.

Our approach can be used in 2D and 3D. Practically, temperature measurements
are performed on 2D scans because of time constraints. In this case we make the
approximation, that the background is harmonic w.r.t. the partial derivatives
in the imaging plane. 

 

\section{Algorithm}
For the following discussion all functions defined on a continuous domain are
non-bold letters while discretely defined functions are denoted as bold
letters. The bold letter variables can always be interpreted as vectors by
reordering. 

The phase, $\Phi(x,y)$, of the measured MRI signal at position~$(x,y)$ is modeled as a
function defined on a continuous domain 
$$
\Phi(x,y) = \Phi_b(x,y) + \Phi_h(x,y), 
$$
where $\Phi_h(x,y) \in (-\pi; \pi]$ is the phase change due to structures of
interest (e.g., heating or susceptibility differences)  and $\Phi_b(x,y) \in
(-\pi; \pi]$ is the background phase due to $B_0$ field inhomogeneities.
Practically we are limited to pixel or voxel based evaluations of $\Phi$,
$\Phi_b$ and $\Phi_h$. The underlying continuously defined functions are used
to derive the method but usually remain unknown. To separate $\Phi_b$ and
$\Phi_h$, a common assumption is that $\Phi_b$ has slow spatial variations
and $\Phi_h$  has large
spatial variations, for example as measured by the magnitude of the gradient.
However, it can be difficult to determine if large gradients  of $\Phi$ are due
to large variations of $\Phi_h$ or the background. Thus, a tunable approach is
needed to separate the two. 

\begin{figure}[h!] \includegraphics[width=9cm]{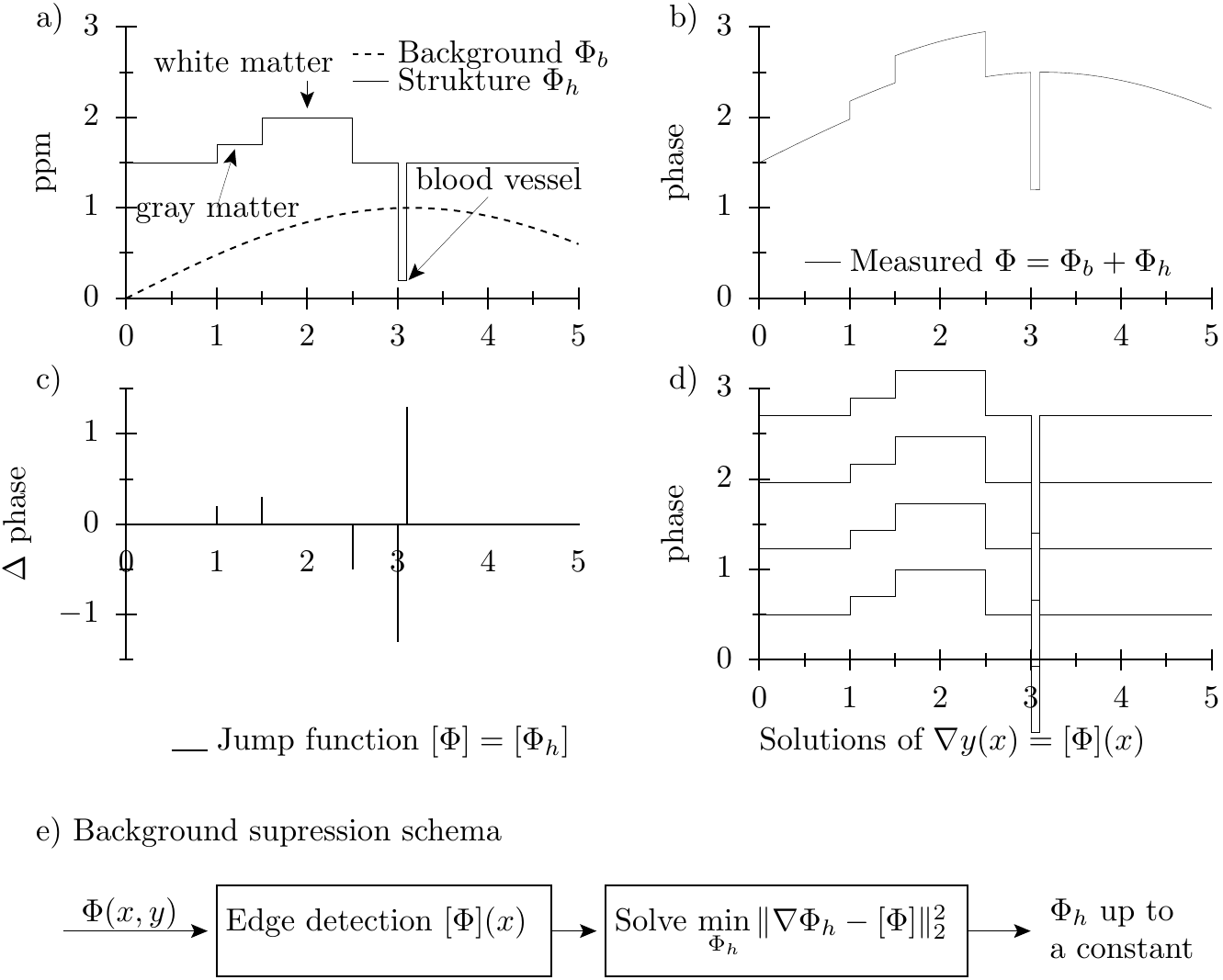}
\caption{Illustration of the proposed background suppression method. Panel a) shows the desired phase signal due to different structures $\phi_h$ and the background signal $\phi_b$. Panel b) shows the measured signal $\phi$. The common jump function (i.e. the hight of the jumps) of $\phi$ and $\phi_h$ is shown in Panel c). The desired $\phi_h$ can be obtained from $[\phi]$ up to a constant as shown in Panel d). Panel e) illustrates the proposed method to recover $\phi_h$ from $\phi$.} \label{fig:surpess_cartoon}
\end{figure}

To illustrate the idea behind our approach, 
assume for simplicity, $\Phi_h$ is
a piecewise constant function. Further, assume that $\Phi_b$ is a smooth
function (i.e., derivatives of any order are bounded). Panel a) of Figure~\ref{fig:surpess_cartoon} shows an example.  Then $\Phi_b$ can be
estimated from $\Phi$ as illustrated in Panel e): First, find the jump size and direction (up
or down) of the discontinuities (Panel c). Because $\Phi_b$ is
smooth and does not contain jumps, all jumps in $\Phi$ are due to jumps in
$\Phi_h$. Further, as a piecewise constant function, $\Phi_h$ is defined by
it's jump discontinuities up to an additive constant as illustrated in Panel d).
I.e. if we find a function
with the same jump discontinuities in size and direction as $\bf\Phi$ which has
zero gradient away from edges, then, this function is equal to $\bf\Phi_h$ up to an
additive constant. This constant can, for example, be determined if we assume
that $\Phi_h$ is zero in a particular region of the image, for example in
regions outside of the heating zone. Phase wrapping artifacts can easily be
removed if  jumps larger or equal to $2\pi$ are projected back to the interval
$[0, 2\pi)$.

\subsection{Edge detection}
We limit the description of the algorithm to 2D functions, more dimensions are
treated analogous to the 2D case.  Two main approaches have been proposed in
the literature to detect edges. One approach is to characterize edges in a
pixel-based image by identifying regions of sharp variations from differences
between neighboring pixels, e.g., the Canny edge detector \cite{Canny:1986uq}.
Another approach uses a wavelet representation of the image to identify the
edges, see e.g., \cite{Mallat92}. However, both approaches cannot easily
distinguish sharp edges from steep gradients \cite{GelbTadmor06}, which in our
case, leads to an insufficient removal of the background. Additionally, $\Phi$
does not contain real jump discontinuities because almost all natural images, and the considered MRI phase images in particular,
are corrupted by some degree of blurring. Therefore a useful edge detection for
this application has to allow sharp gradients to be identified as edges, while
dismissing less sharp gradients as background. Here, we propose a method that
includes a tuning parameter to achieve an appropriate separation of background
phase and phase of interest.

The tuning parameter has by be chosen on a case to case basis by visually inspecting
the images. We have found that for similar applications, such as removing the background 
field in brain MRI scans, similar parameters can be used. 


To make the edge detection more robust, we assume only very few pixels contain
sharp gradients, i.e. the gradients are {\it sparse}. This assumption is
similar to the sparsity assumption that is usually made in a compressed sensing
approach \cite{Lustig:2008fk}.  Therefore, we choose the approach in
\cite{Stefan:2012kq} which is sparsity enforcing and robust in the presence of
steep gradients. The proposed method is similar to the method described in
\cite{Stefan:2012kq}, with the difference, that we exchange the edge detection
kernels that are based on partial Fourier from the original publication with a
kernel based on finite differences as described as follows. 

Analogous to \cite{Archibaldetal05}, edges in $\Phi(x,y)$ are represented by
the so called {\it jump functions} in $x$- and $y$- direction denoted by
$[\Phi]_x(x,y)$ and $[\Phi]_y(x,y)$. The jump functions are defined by 
$$
\begin{array}{rcl} 
  \left[\Phi\right]_x(x,y) & = & \Phi(x^+,y) -   \Phi(x^-,y) \\ 
  \left[\Phi\right]_y(x,y) & = &  \Phi(x,y^+) -  \Phi(x,y^-) \\  
\end{array}
$$ 
where $x^+$, $y^+$,$x^-$ and $y^-$ are the well defined right and left hand
limits. If $\Phi(x,y)$ is smooth with respect to $x$ and $y$ at $(x,y)$, then
$[\Phi]_{\{x,y\}}(x,y)=0$. Further, $[\Phi]_x(\xi,\nu)$ or $[\Phi]_y(\xi,\nu)$
is the height of the jump if $\Phi$ has a jump in $x$- or $y$-direction at
$(\xi,\nu)$ respectively. Our goal is to find discrete representations of the
jump functions. 

Assume a pixelated image is given by sampling a function defined on a
continuous domain on a regular grid, i.e.,  $({\bf \Phi})_{i,j} =
\Phi(x_i,y_i)$, where $i,j = 1, \cdots, N$. Since the underlying continuously
defined function $\Phi$ is not known, finding an approximate for reliable edge
detection is very difficult. In \cite{Archibaldetal05,GelbTadmor06} the
underlying function is expressed as a Taylor expansion around a given pixel
location. The remainder of a Taylor expansion up to degree $m$, vanishes if the
underlying function is a polynomial of degree $m$ in the neighborhood of the
given pixel. Similarly, if the underlying function is not a polynomial, but is
sufficiently smooth, i.e., has several continuous derivatives, the remainder
can be expected to be very small. On the other hand, it is very large at the
locations of jump discontinuities. If the Taylor expansion is approximated by
the interpolating polynomial around a given grid point, this idea leads to the
so called {\it polynomial annihilation} edge detection method, described in
\cite{Archibaldetal05}, which can be expressed as \begin{equation}
\begin{array}{rcl}
  (\bf g_x^m)_{i,j} & = & \sum\limits_{k=1}^{m+1} c_{m,k} \Phi(x_{i+k},y_j)\\
  (\bf g_y^m)_{i,j} & = & \sum\limits_{k=1}^{m+1} c_{m,k} \Phi(x_i, y_{j+k})
  \label{eq:def:g} 
\end{array} 
\end{equation} 
where $\bf g_x^m$ and $\bf g_y^m$ are the jump function approximations for a
$m^\textrm{th}$ order polynomial edge detector, and 
\begin{equation} 
c_{m,k} =
-  \frac{ m! q_m}{\Pi_{l=1,k \neq j}^{m+1} (k-l)}.  \label{eq:def:c_j}
\end{equation} 
The scalars $q_m$ is chosen such that the jump height is preserved.  A more
detailed discussion and derivation of the edge detection kernels $c_{m,k}$ can
be found in \cite{SRG:08}. From here on we fix the order $m$, if not indicated
otherwise, and drop it from $\bf g$ and $\bf c$ for clarity. The sums in
(\ref{eq:def:g}) can be computed by a convolution:
\begin{eqnarray} 
\bf g_x & =  & \bf c * \bf \Phi \nonumber \\ \bf g_y & = & \bf
c^T * \bf \Phi, \label{eq:jump_conv} 
\end{eqnarray} 
where $\bf c$ is obtained from (\ref{eq:def:c_j}) by zero padding: 
\begin{equation} \bf c_{i,k} = \left\{
\begin{array}{rcl} c_{m,k} & \textrm{for $i =1$ and $k \leq m +1$} \\ 0 &
\textrm{otherwise} \end{array} \right. , \label{eq:jump_detector}
\end{equation} and $\bf c^T_{i,j} = \bf c_{j,i}$.

\begin{figure}[h!] \includegraphics[width=9cm]{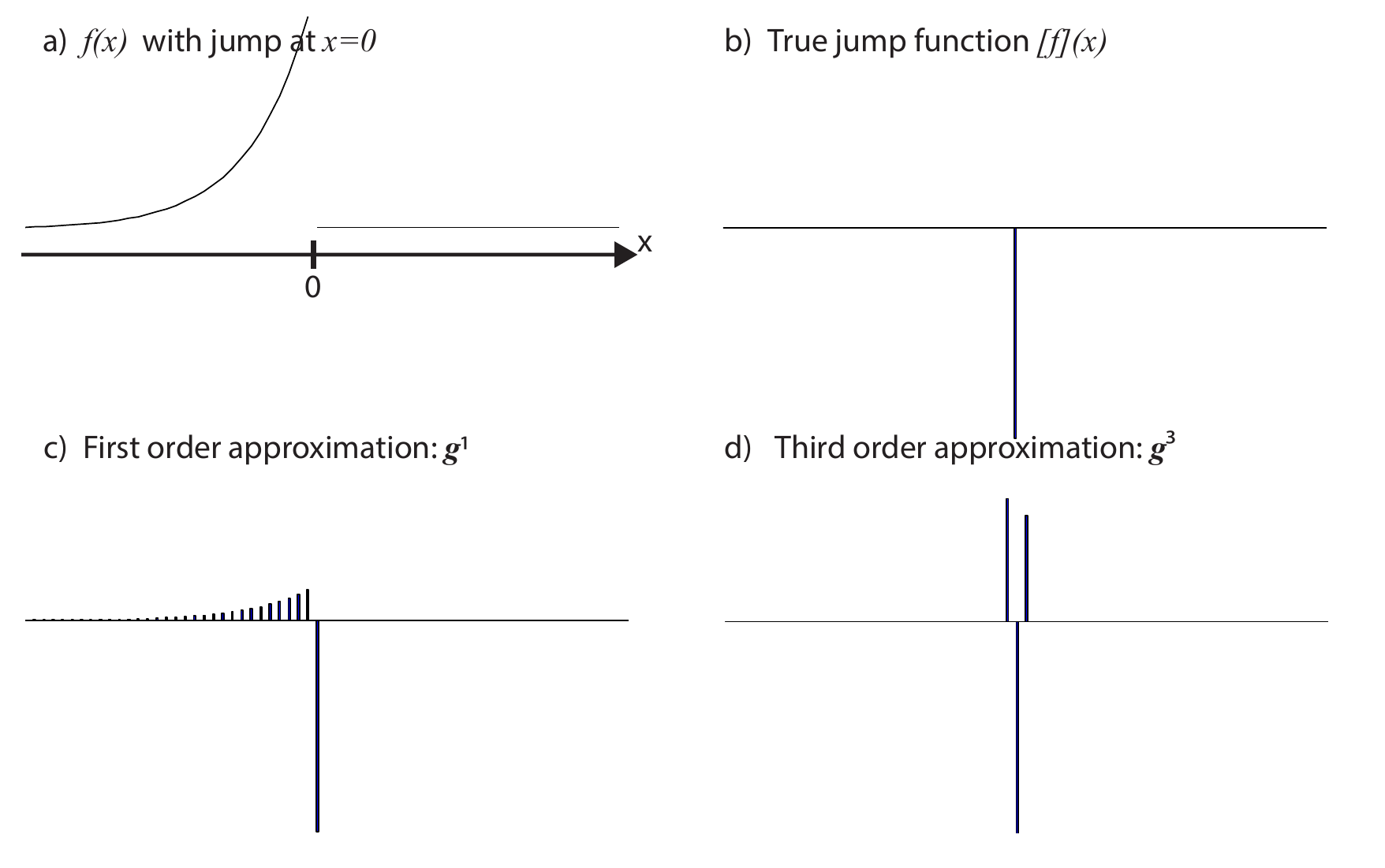}
\caption{Illustration of the jump function (b)  of a piecewise smooth function
(a). The approximations using low order (c) and high order (d). Lower order (c)
can lead to miss classification of edges in regions of large gradients. Higher
order (d) produces unwanted oscillations.} \label{fig:oscillations}
\end{figure}

It is shown in \cite{Archibaldetal05,GelbTadmor06} that (\ref{eq:def:g})
converges to the true jump function as $N \to \infty$, however, this method
introduces oscillations near locations of edges. These oscillations can easily
lead to misclassification of edges. Figure~\ref{fig:oscillations} shows a test
function $f(x)$ in panel a and its true jump function $[f](x)$in panel (b).
Panel (c) shows the discrete approximation of the jump function using $m=1$
(${\bf g}^1$). The non-zero components of the jump function approximation can
easily be misinterpreted as edges. Panel (d) shows the discrete approximation
of the jump function using $m=3$ (${\bf g}^3$). The oscillations around the
jump are artifacts of the higher order method.

To reduce the oscillations a non-linear filtering technique has been used which
combines the jump function approximations of several orders to one combined
jump function. However this approximation is very sensitive to noise
\cite{Archibaldetal05,GelbTadmor06}. In \cite{Stefan:2012kq} it was proposed to
remove similar oscillations by deconvolution.  Oscillations are modeled as a
Point Spread Function (PSF), the so called {\it matching waveform}, and removed
using $\ell 1$ minimization, which forces the edges to be sparse. In
\cite{Stefan:2012kq} a closed form solution of the matching waveform is used to
perform the deconvolution. Because the matching waveform based on
(\ref{eq:def:g}) has no closed form it has to be estimated. An estimate can be
obtained by applying (\ref{eq:def:g}) to a unit jump, i.e. a jump in either
$x-$ or $y-$ direction with size one. This results in the matching wave form
\begin{equation}
    \bf w_{i,j} =  \left\{ \begin{array}{cl} \sum\limits_{k=1}^{N/2} \bf c_{k,j} +  \sum\limits_{k=N-i+2}^{N} \bf c_{k,j} & i = 1, \dots, N/2 \\
     \sum\limits_{k=N-i+2}^{N/2} \bf c_{k,j} &  i = N/2+1, \dots, N 
    \end{array} \right. ,
    \label{eq:matching_wave_formular} 
\end{equation}
for edge detection kernels of order $m$.

Let $\bf u_x$ be the true jump function in $x-$ direction, and $\bf u_y$ be the
true jump function in $y-$ direction. Given the matching wave form $w$ we
assume $\bf g_x \approx {\bf w} * \bf u_x$ and $\bf g_x   \approx {\bf w^T} *
\bf u_y$, where $\bf u_x$ and $\bf u_y$ are the desired jump function
approximations. Thus, using (\ref{eq:jump_conv}), approximations of $\bf u_x$
and $\bf u_y$ can be found by solving the optimization problem 
\begin{eqnarray}
  ({\bf u_x},{\bf u_y}) & = & \arg \min_{({\bf u_x},{\bf u_y})}   \frac{1}{2} \| {\bf w} * {\bf u_x} - {\bf c} * {\bf \Phi} \|_2^2  \nonumber \\
  & & + \frac{1}{2} \| {\bf w}^T * {\bf u_y} - {\bf c}^T * {\bf \Phi} \|_2^2 \nonumber \\
  & & + \lambda ( \| ({\bf u_x},{\bf u_y}) \|_{1,2}), 
  \label{eq:edge_solution}
\end{eqnarray}
where $\| ({\bf u_x},{\bf u_y}) \|_{1,2} = \sum_{i,j} \sqrt{ ((\bf
u_x)_{i,j})^2 + ((\bf u_y)_{i,j})^2}$ and $\lambda$ is a regularization
parameter. In the appendix a proof of convergence of the proposed method to the
true jump function as $N \to \infty$ similar to the results in
\cite{Archibaldetal05,GelbTadmor06} is presented.

The measured phase $\bf \Phi$ contains jump discontinuities of size $2\pi$ due
to phase wrapping. With the assumption that there are no jumps larger than
$2\pi$ due to structure, it is easy to identify jumps from phase wrapping. To
remove all jumps from phase wrapping we replace $\bf u_x \to ( {\bf u_x} \mod
2\pi ) $ and  $\bf u_y \to ({\bf u_y} \mod 2\pi)$. To solve the optimization
problem an augmented Lagrangian method can be used, e.g., see \cite{Wu2010}.

\subsection{Phase image reconstruction}
 In the second stage a convex optimization problem is solved to find a function
such that its gradients match the detected sharp gradients of the phase.
 
 \begin{equation}
        \min\limits_{\bf \Phi_h} \| \nabla {\bf \Phi_h} - (\bf u_x, \bf u_y) \|_{\bf{m}}^2  + \epsilon \| \bf \Phi_h \|_2^2,
        \label{eq:im_recpn}
 \end{equation}
 where 
 $$
 \| \bf \xi \|_{\bf{m}}^2  =  \sum\limits_{i,j} \bf{m}_{i,j} | \bf \xi_{i,j} |^2,
 $$
 and $\bf{m}$ is chosen such that unreliable phase measurements are
down-weighted. We use the square of the magnitude of the image as weights
because it has previously been connected to noise in the phase
\cite{Taylor:2008}. The second part of (\ref{eq:im_recpn}) is needed (i.e.,
$\epsilon>0$) because $\nabla {\bf \Phi_h}$ is invariant to an additive constant. We
choose $\epsilon$ large enough to make the problem numerically stable but small
enough that the regularization does not change the gradient of the solution
significantly ($\epsilon \approx 10^{-8}$). 
 
For a given ROI that does not contain edges,
the minimizer of (\ref{eq:im_recpn}) automatically satisfies the Laplace
equation $\vec{\nabla}^2 \Phi_h = 0$. Note that as a
consequence the algorithm separates the phase in a harmonic part (away from
edges) and a non-harmonic part at the edges. The background phase can be found
by subtracting $\Phi_h$ from $\Phi$. 

Because the edges are removed the background $\Phi_b$ is a continuous function, it is not necessarily harmonic because it may contain discontinuities in higher derivatives. The method could be expanded to detect and remove discontinuities in higher derivatives, see e.g. \cite{saxenaetal2007}, however the noise content in the images may be a problem since edge detection on gradients are very sensitive to noise.

The background suppression procedure is summarized in
Algorithm~\ref{alg:bg_supress_pseudo}.

\begin{algorithm}
\caption{Background suppression algorithm overview}
\label{alg:bg_supress_pseudo}
\begin{algorithmic}[1]
\REQUIRE A complex MRI image $\bf f \in \mathcal{C}^{N_\textrm{pixel}}$
\ENSURE $\bf \Phi_h$. \\
\begin{enumerate}
  \item Find sharp gradients ($\bf u_x, \bf u_y, \bf u_z$) in the phase of $\bf f$ using sparsity enforcing edge detection.
  \item Solve $\min_{\bf x} \| \nabla {\bf x} - (\bf u_x, \bf u_y, \bf u_z) \|_{\bf m}^2 + \epsilon \| \bf x \|_2^2$ with weighting $\bf m=| \bf f |^2$ 
  \item Find $a \in \mathcal{R}$, such that $({\bf x})_{i,j,k} + a = r$ for a reference voxel, i.e. $r={\bf \Phi}_{i,j,k}$
  \item $\bf \Phi_h := \bf x + a$ 
\end{enumerate}
\end{algorithmic} 
\end{algorithm}

\section{Results and Discussion}

We apply the proposed method to  MR temperature imaging and $\Phi_b$ correction of
PRF measurements. 

%

\subsection{Reference-less temperature imaging}
\begin{figure*}[h!t]
\centering
\subfloat[Magnitude image]{\includegraphics[width=2.4in]{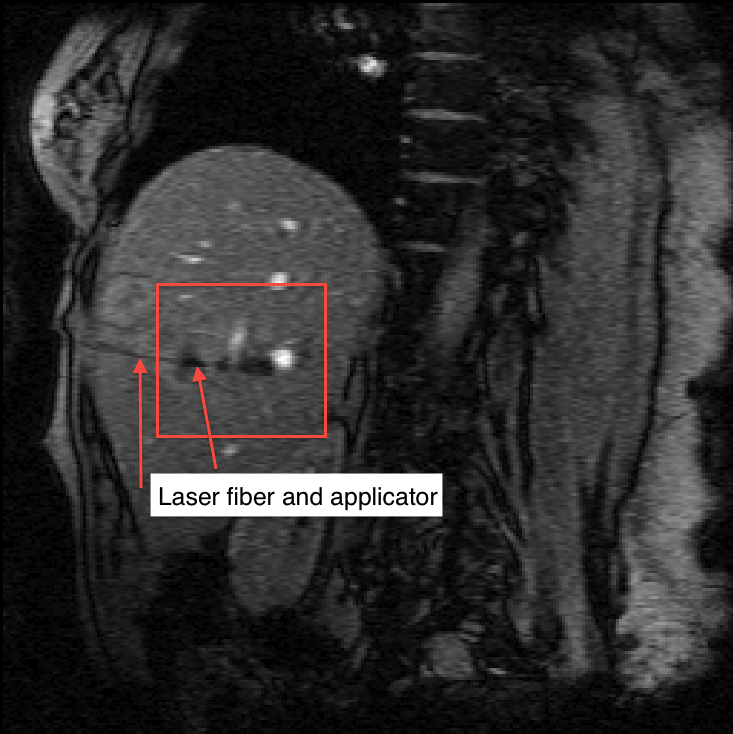}%
\label{fig:heat1:mag}}
\subfloat[Phase image]{\includegraphics[width=2.4in]{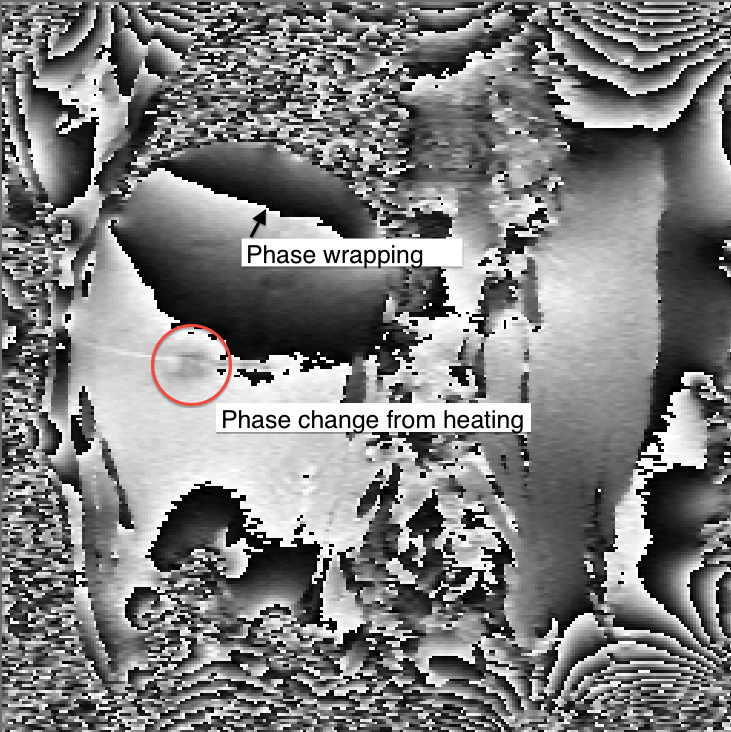}%
\label{fig:heat1:phase}} \\
\subfloat[First order jump function approximation]{\includegraphics[width=2.4in]{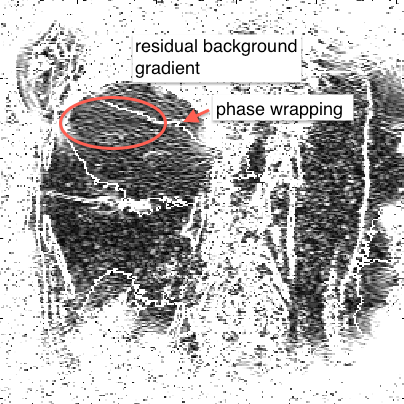}%
\label{fig:heat1:edge_o1}}
\subfloat[Third order jump function approximation]{\includegraphics[width=2.4in]{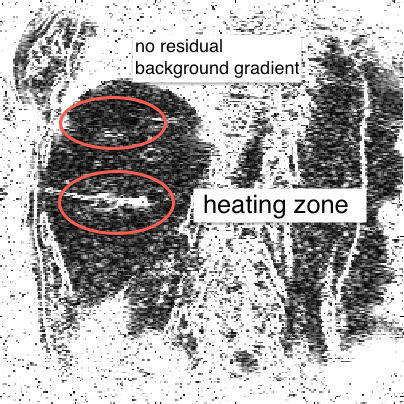}%
\label{fig:heat1:edge_o3}}

\caption{Magnitude and phase of a gradient echo pulse sequence during thermal
therapy. The heating causes signal loss in the magnitude image and a shifted
phase in the phase image. The phase change due to heating and phase wrapping artifacts are pointed out.
The images of Figure~\ref{fig:heat2} are zoomed to the box indicated in (a). Panels (c) and (d) show the magnitude of the jump function of (b) computed using eq. (\ref{eq:edge_solution}) with a first and third order kernel. The first order edge detection clearly shows a residual background shading on the top of the liver, while the third order detector does not show the shading. The phase wrapping jumps in (c) can be easily removed by projecting the jump function to an interval of $[-\pi,\pi)$, which was done in (d).} \label{fig:heat1}
\end{figure*}

\begin{figure}[h!t]
\centering
\includegraphics[width=3in]{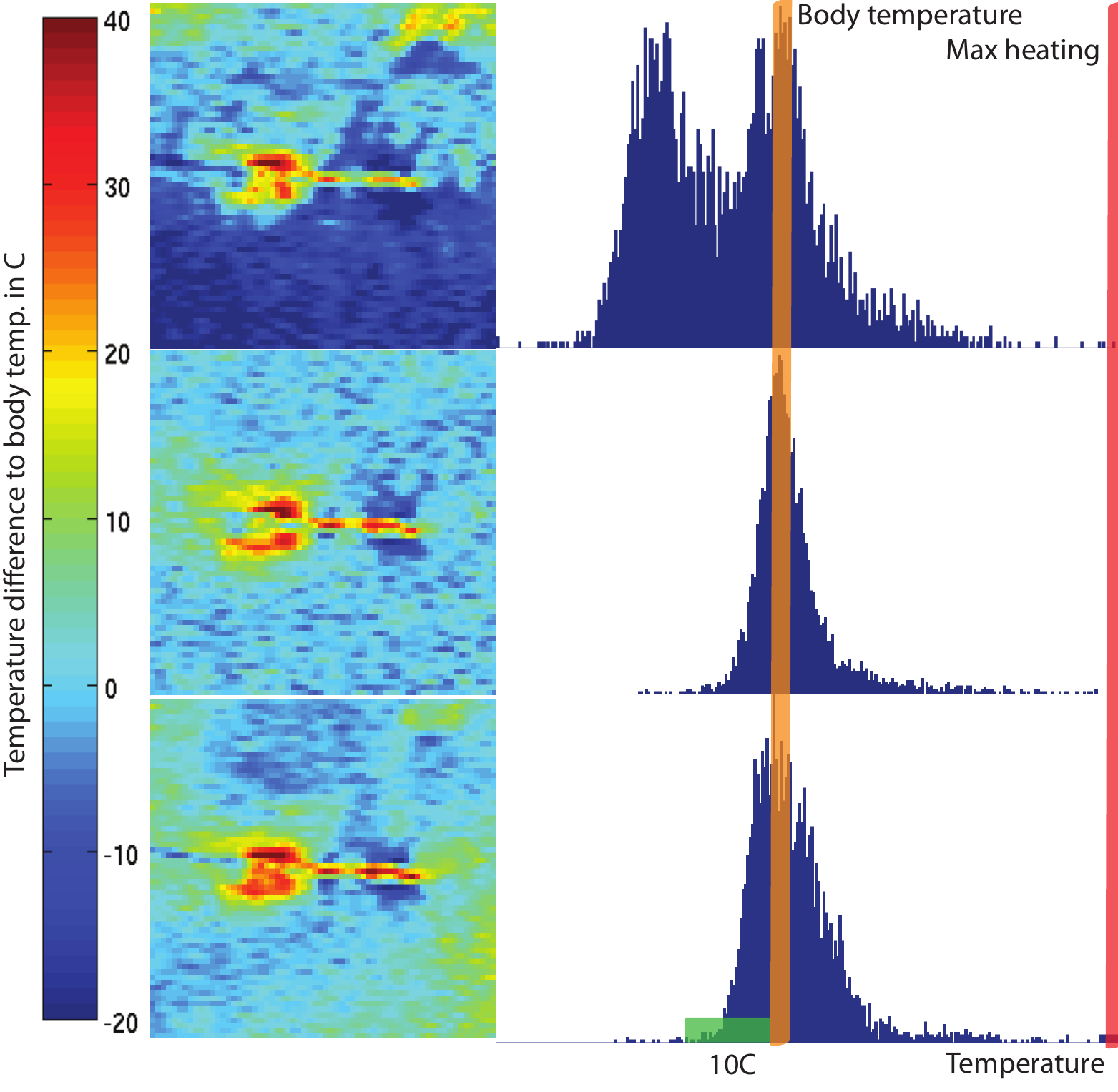}
\caption{Comparison of reference frame subtraction (top and middle) and the
proposed method with patient motion (e.g. due to breathing) between the reference frame and the heating frame. 
The top and middle rows show a case with and without motion.
The bottom
row shows the opposed background suppression technique, which does not use a
reference frame. The histogram illustrate the temperature distribution in the
images relative to the body temperature. } \label{fig:heat2}
\end{figure}

For temperature imaging, complex valued 2D images are collected using a
gradient echo pulse sequence. Acquisition parameters: Spoiled gradient recalled
sequence, $B_0=1.5T$, $TR=27ms$, $TE=12ms$, Pixel Bandwidth $280Hz$, Matrix
$256\times128$. 

Figure~\ref{fig:heat1}(a) and (b) shows the magnitude and phase of the acquired image
during a laser induced thermal therapy. The dark regions inside the box cause a signal loss
due to the heating. At the corresponding location in the phase image a negative
shift of the phase due to the heating can be seen. This shift is linear with
temperature change and is used to measure temperature changes. The phase
variations outside the heating zone are off resonance effects from an
in-homogeneous $B_0$ field, i.e. $\Phi_b$. The jumps (white to black) are due to phase
wrapping artifacts. In Panels (c) and (d) two jump function approximations of the phase image 
computed  using (\ref{eq:edge_solution})  with a kernels of order one and three are shown.
Similar to the example in Figure~\ref{fig:oscillations} the first order jump function approximation is non-zero
in areas of large gradients. If the first order approximation is used in (\ref{eq:im_recpn}), 
then the reconstructed image shows the same background that original phase image has. On the other
hand, if the third order approximation is used in (\ref{eq:im_recpn}), then the resulting
reconstruction has a suppressed background phase. The parameter $\lambda$ in (\ref{eq:edge_solution}) 
is used to suppress noise and to force a sparse jump function approximation. Note that (d) is significantly 
more sparse than (c). For this and all other experiments in this paper $\lambda=10^{-6}$ was used.


In traditional temperature imaging the complex image of the heating is
subtracted by an image before the heating. If the background phase did not change
during the heating procedure then the phase of the resulting image only
contains changes due to heating. Thus the procedure is very sensitive to
patient motion and other bulk magnetic field changes. 

Figure~\ref{fig:heat2} shows a temperature map corresponding to the highlighted
region in Figure~\ref{fig:heat1}. The left part of the figure shows the
temperature image and the right part the histograms corresponding to the
images. The temperature maps are  obtained using the reference frame subtraction
method in the top and middle row and the proposed method in the bottom row.  In
the top row a pre-heating frame was used, where the background phase changed as
compared to the heating frame. Temperature measurements show large variations
away from the heating zone in the middle. The temperature histogram on the
right also reflects the large variations. The width of the histogram around the
body temperature is often used to measure the error of the method. In the  the
middle row a pre-heating frame was used where the background did not change
relative to the heating frame. The histogram on the right is narrower around
the body temperature indicating a much smaller measurement error. Unfortunately
often it is very difficult or impossible to find a pre-heating frame with the
same background phase \cite{Vigen:2003aa}.The temperature estimate using the
proposed method is shown in the bottom row. The temperature map on the left
shows a much more homogeneous background, comparable to the top row image. The
histogram is also much narrower around the body temperature. This result demonstrates,
that our method results in similar temperature estimates as the reference frame subtraction. Of course we cannot expect to 
obtain exactly the same solution as a successful reference frame subtraction. 

Temperature variations around the body temperature are less than 10 degC. 
The edge detection methods naturally provide the boundary of the laser applicator
for this application.
Potential difficulties of the proposed method arise if the heating signature is
very smooth, or the background contains large gradients. For the proposed method
a very smooth temperature change will  be absorbed into the background field $\Phi_b$. 
Similarly, large gradients in the background field may be absorbed into $\Phi_h$.

\subsection{Background Removal for quantitive Susceptibility Weighted Imaging}
\begin{figure*}[h!t]
\centering
\subfloat[Magnitude of 1st echo]{\includegraphics[width=2.5in]{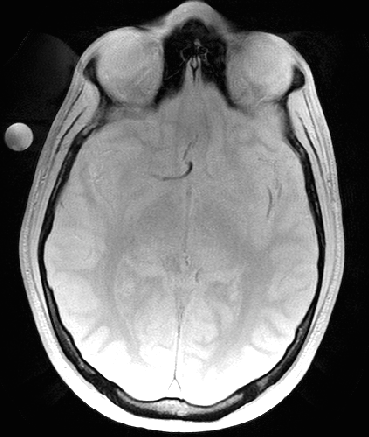}
\label{fig:ppm:mag}}
\hfil
\subfloat[PPM map, $\bf \Phi$]{\includegraphics[width=2.5in]{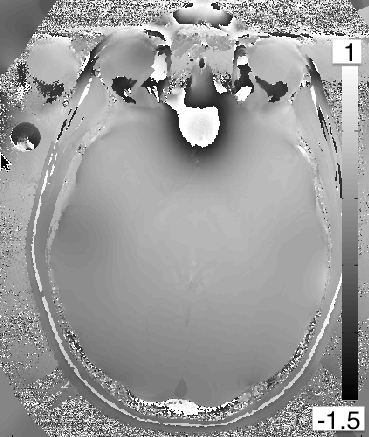}%
\label{fig:ppm:phase}} \\
\subfloat[Corrected PPM map,${\bf \Phi_h}$]{\includegraphics[width=2.5in]{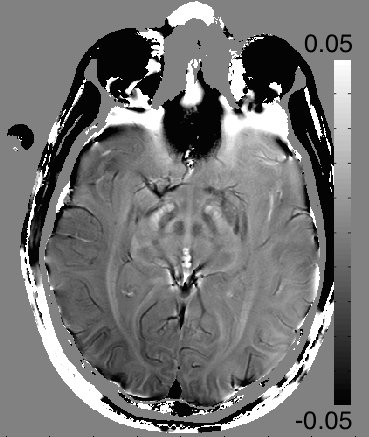}%
\label{fig:ppm_correted:phase}} 
\hfil
\subfloat[PPM map of the estimated background, ${\bf \Phi_b}={\bf \Phi} - {\bf \Phi_h}$]{\includegraphics[width=2.5in]{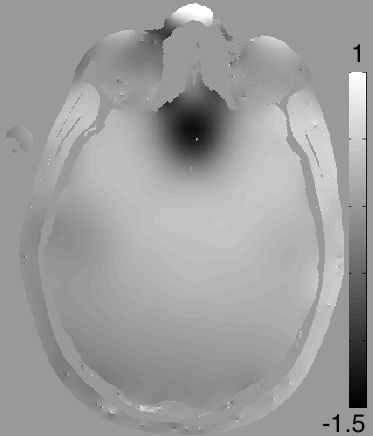}%
\label{fig:ppm:phase}}
\caption{
Background suppression of a ppm map in a healthy volunteer. (a): Magnitude of the first echo of a multi echo sequence. (b): PRF map estimation using ARMA technique.(c): Background suppressed PRF map. Blood vessels appear darker due to magnetic susceptibility effects. (d): Estimated background field map in ppm.}
\label{fig:brain_SWI}
\end{figure*}

Figure~\ref{fig:brain_SWI}  shows brain data, acquired using a 3T scanner
(MR750; GE Healthcare Technologies, Waukesha, WI) from a healthy volunteer,
using a 2D-MFGRE sequence with echo spacing = 3.4ms, number of echoes = 16, TE
= 2.4 to 28.5 ms, TR = 2200 ms, flip = 60, FOV = 22cm, matrix = 320x256, slice
thickness = 5mm, number of slices = 30. An ASSET calibration scan is performed,
and sensitivity maps for each of the 32 coils in a 32-channel head coil are
obtained. K-space is under-sampled by a factor of 2 by skipping every other
frequency encoding line. The magnitude of the first echo is shown in panel (a).
The reconstructed images are processed pixel wise by fitting a infinite impulse
response (IIR) filter to each echo train. This was done using the Auto
Regressive Moving Average (ARMA) \cite{Taylor:2008} technique. The PRF of the
dominant peak  can be extracted from the coefficients of the IIR filter. The
resulting PRF in ppm at each pixel is illustrated in Panel (b). Panel (c) shows
the background corrected image. Blood vessels appear darker due to a negative
ppm shift of oxidized blood. Brain structures in the mid brain also show
significant contrast after applying our background suppression technique. Panel
(d) shows the difference of (c) and (b), i.e. the suppressed background. Off
resonance shifts can be seen in particular near air cavities near the sinus and
the ears. This example demonstrates the potential of the method for quantitive
estimation of PRF shifts due to magnetic susceptibility differences for
different tissue types. 

\subsection{Noise sensitivity of the background map estimation}
\begin{figure*}[h!t]
\centering
\subfloat[Magnitude image]{\includegraphics[width=2.5in]{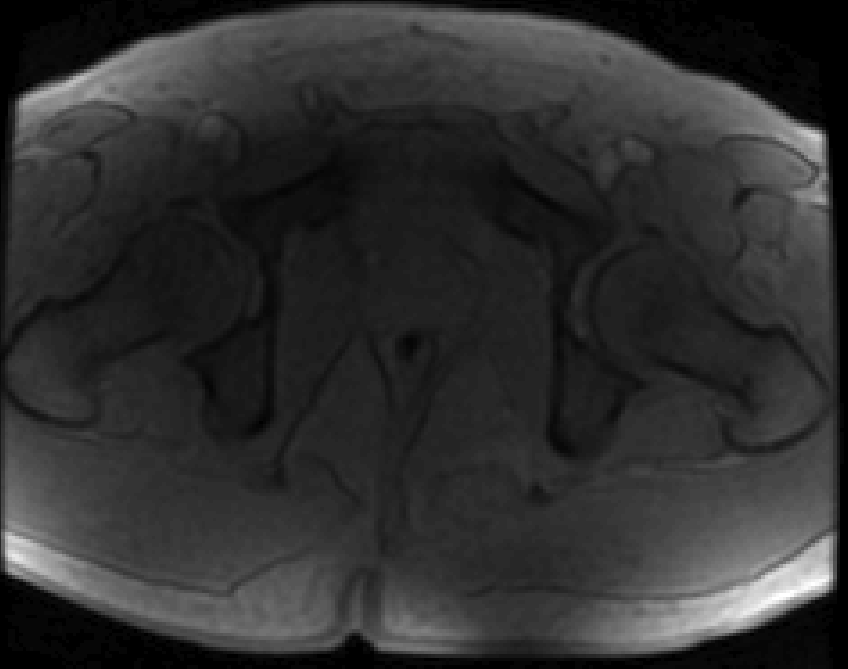}%
\label{fig:ppm:mag}}
\hfil
\subfloat[PPM map, {\bf Phi}]{\includegraphics[width=2.5in]{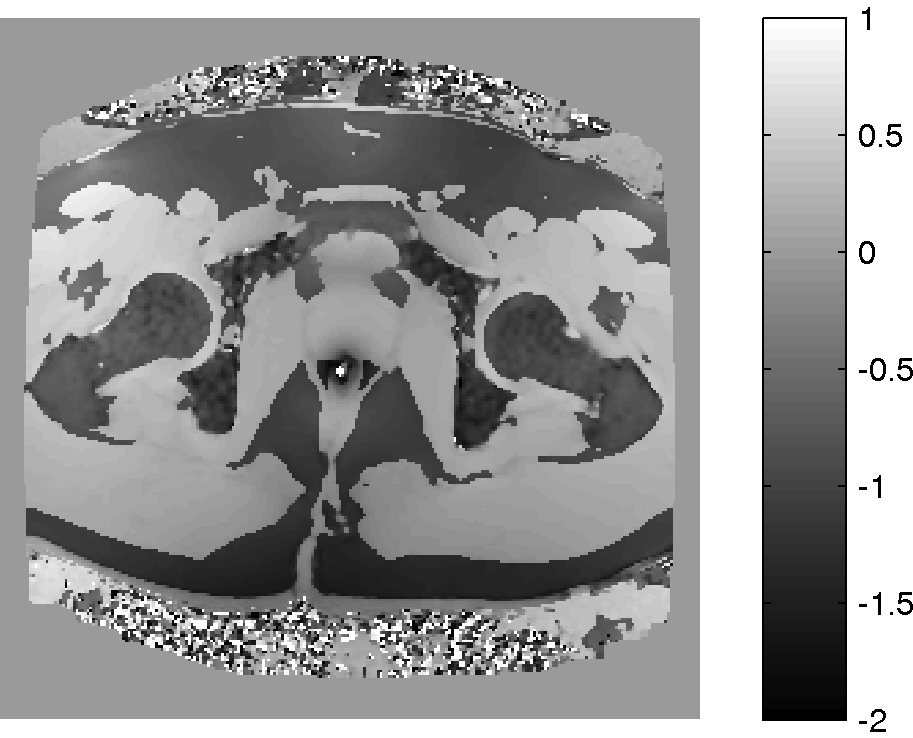}%
\label{fig:ppm:phase}} \\
\subfloat[Corrected PPM map, $\bf \Phi_h$]{\includegraphics[width=2.5in]{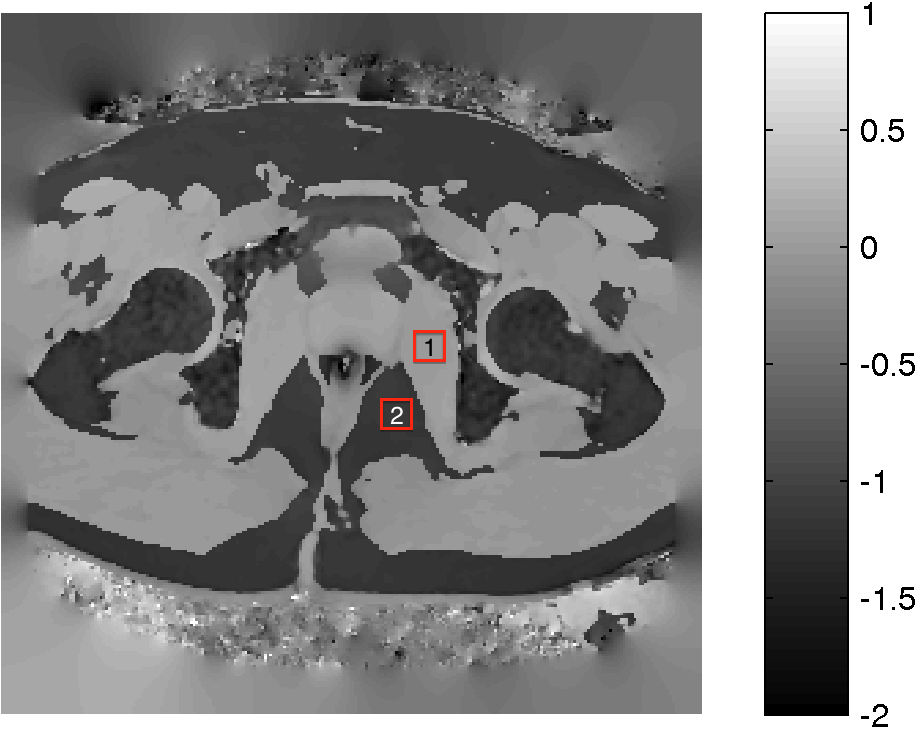}%
\label{fig:ppm_correted:phase}} 
\hfil
\subfloat[$\bf \Phi_b$ field map ]{\includegraphics[width=2.5in]{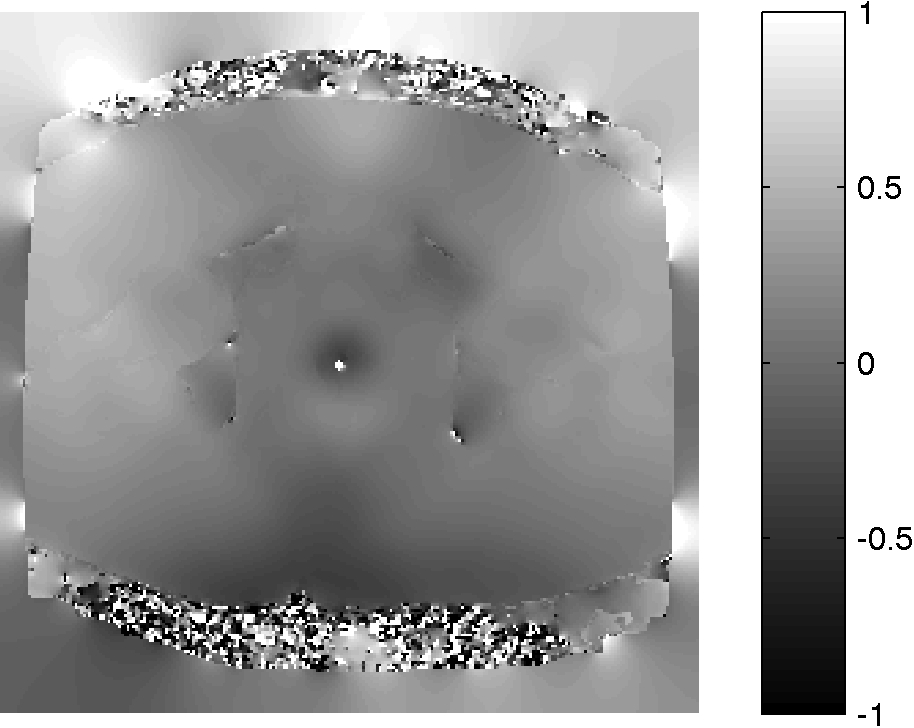}%
\label{fig:ppm:phase}}
\caption{Background suppression of a ppm map with water and fat signal. (a):
Magnitude of the first echo of a multi echo sequence. (b): PRF map estimation
using ARMA technique.(c): Background suppressed PRF map. (d): Estimated
background field map in ppm. The fat and water peak is aliased with $4.32$ ppm
due to the chosen echo spacing. 
the fat-water boundary.  The labels 1 and 2 are water and fat ROIs that are used to determine the CNR gain of the method.}
\label{fig:ppm}
\end{figure*}

In this section we add Gaussian noise of various levels to a chemical shift
image with high contrast to noise ratio (CNR) to test the robustness under
noise. 

An MFGRE sequence was used to collect using a 3T scanner (SIGNA; GE Healthcare Technologies, Waukesha, WI)  
16 echoes with echo spacing of $1.816ms$,
$TR=75ms$, Matrix $128\times128$ and Pixel Bandwidth of $651.016 Hz$.
Figure~\ref{fig:ppm} (a) illustrates the magnitude of the first echo of a slice
through the pelvis. The multi echo data is processed pixel wise by the ARMA
technique. The resulting PRF in ppm at each pixel is illustrated in Panel (b).
The dark areas in panel (b) are subcutaneous lipid.  Inhomogeneities of the PRF
can be seen near the air inclusion in the middle and throughout the image, in
particular in the lipid regions. The result of the background suppression
technique is shown in Panel (c), which illustrates, that most of the background
inhomogeneities have been removed. Note the fat-water ppm difference is about
$1.16 ppm$ and not the expected shift of about $3-3.5ppm$. This is a
consequence of aliasing introduced by the finite sampling in echo direction. If
the aliasing is taken into account by adding the bandwidth of $4.3212 ppm$,
then the fat-water shift is $3ppm$ as expected. The pixel wise difference of
the images in Panels (b) and (c) is shown in Panel (d). Note that the estimated
background field ${\bf \Phi_b}$ is smooth in areas of transitions of fat and water. The map is
not necessarily smooth in areas of low signal, however a correct estimation of
the background in low signal areas is not a well-posed problem.

Noise is added to the original ppm map and a contrast to noise ratio is
computed as follows: Two regions of interest (ROIs) are selected $R_w, R_f$,
one in water the other in fat, labeled 1 and 2 in Figure~\ref{fig:ppm} (c).
The size of the ROI is $11\times11$ pixels. The
CNR is computed as 
$$
        CNR = \frac{\textrm{mean}(R_w)-\textrm{mean}(R_f)}{\textrm{std}(R_w \bigcup R_f)}.
$$  
\begin{figure}[h!t]
\centering
\includegraphics[width=3in]{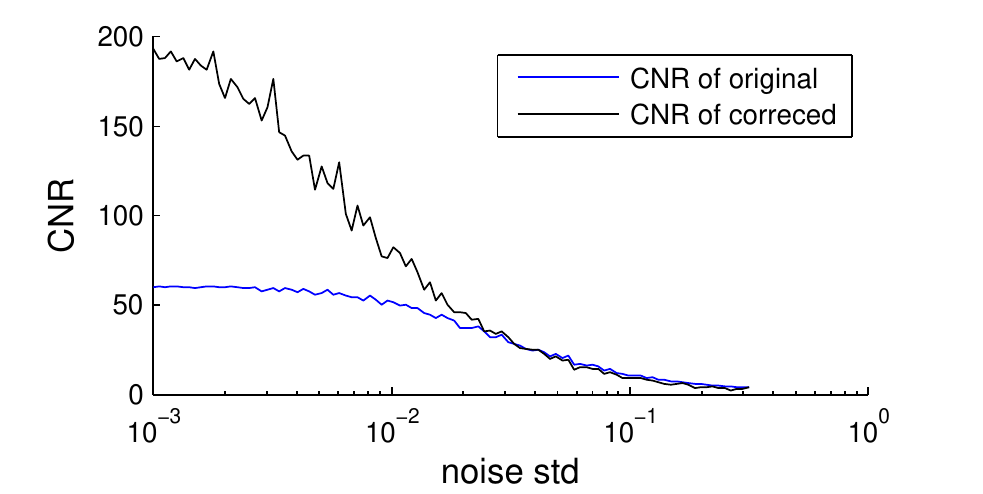}
\caption{Contrast to noise ratio for different levels of noise is shown. The
uncorrected image has a low CNR even at small noise levels because of the
inhomogeneity of the background. For larger noise levels the CNR of the
background suppression is reduced because the edge detection step becomes less
well posed. However, the CNR is on average is larger than the CNR of the
original data.} \label{fig:CNR}
\end{figure}

Figure~\ref{fig:CNR} shows the CNR computed for different noise levels. For
very low noise levels the CNR for the uncorrected ppm map is about $60$, for
the corrected map over $170$. This is because the inhomogeneities in the ROIs
cause a large standard deviation thus a small CNR for the uncorrected image.
With increasing noise the edge detection becomes more difficult and the
background suppression method results in lower CNR gains. However the method
does not amplify the noise and gives an significant CNR improvement for low and
moderate noise levels.

\section{Conclusion}
We have presented a novel background suppression method for MRI phase data. The
method works in two stages. First edges are detected in the original phase
image using a higher order edge detection method that is based on compressed
sensing principles and thus promotes a sparse edge map. In a second stage the
corrected phase map is computed by finding a piecewise harmonic function that
fits the edges. 

We have demonstrated in numerical experiments that the method can be used to
remove the background inhomogeneities of gradient echo MRI sequences.
Applications include MR temperature imaging and chemical shift imaging. The
method is computationally efficient and can remove the background of a
$256\times256$ image a few seconds using our non-optimized MATLAB code. 

In future work we want to explore other application where inhomogeneities have
to be removed and optimize our code. We also would like to improve the
performance of the method in regions of very large but smooth gradients in
particular near tissue air interfaces.

\section*{Acknowledgment}
This research is supported in part by the MD
Anderson Cancer Center Support Grant CA016672 and the National
Institutes of Health (NIH) award 1R21EB010196-01 and Apache Corporation.

The authors would like to thank Wotao Yin at UCLA for many useful discussions
while he was the Post-doc advisor for the first author.
\section*{Acknowledgment}
This research is supported in part by the MD
Anderson Cancer Center Support Grant CA016672 and the National
Institutes of Health (NIH) award 1R21EB010196-01 and Apache Corporation.

The authors would like to thank Wotao Yin at UCLA for many useful discussions
while he was the Post-doc advisor for the first author.


%

\appendix

\section{Detection guarantees}
Because the proposed method is based on convex optimization in the context of
compressive sensing, we can derive some results when such an edge detection
procedure can be expected to be successful.

For simplicity we first consider the one dimensional case, with edge detection
kernels of the form (\ref{eq:def:c_j}) with no noise. In this case
\eq{eq:edge_solution}  can be simplified, and renormalized, such that
\begin{eqnarray}
\by  & =  & \arg \min\limits_{\bu} \{ \frac{1}{2} \|\bu * \bw - \bc * {\bf \Phi} \|_2^2 + \lambda \| \bu \|_1^1 \}\\
       & =  & \arg \min\limits_{\bu} \{ \frac{1}{2} \|W \bu  -  \bb \|_2^2 + \lambda \| \bu \|_1^1 \} \\
       & =  & \arg \min\limits_{\bu} \{ \frac{1}{2} \|\frac{1}{\| \bc \|_2}W \bu  -  \frac{\bb}{\| \bc \|_2} \|_2^2 + \frac{\lambda}{\| \bc \|^2_2} \| \bu \|_1^1 \},
\label{eq:deconb_edge_1d}
\end{eqnarray}
where $W$ is the unique Toeplitz matrix such that $W\bx = \bx * \bw$ for an
arbitrary vector $\bx$. Furthermore, $\bb=\bc * {\bf \Phi}=L{\bf \Phi}$ and $L$
is a Toeplitz matrix such that $L\bx = \bc * \bx$ for arbitrary $\bx$. The
entries of $\bc$ are given by $\eq{eq:jump_detector}$ and the entries of $\bw$
by \eq{eq:matching_wave_formular}.

A key concept to derived reconstruction guarantees in compressive sensing is
the concept of the restricted isometry property (RIP). A matrix $A$ has the RIP
if, for an arbitrary choice of an index set $\bt$ with cardinality less or
equal to the sparseness $S$ of the true solution (in other words $|\bt| \leq
S$), there exists a constant $\delta_S$, such that \begin{equation}
        (1-\delta_S) \|\bx\|_2^2 \leq \| A_\bt \bx \|^2_2 \leq (1+\delta_S) \|\bx\|_2^2,
\end{equation}
for an arbitrary vector $\bx$. Here $A_\bt$ denotes the sub-matrix of $A$,
obtained by extracting the columns corresponding to the indices in $\bt$, see
e.g. \cite{Candes05}.

It is difficult to derive the RIP for an arbitrary selection of columns of $W$,
however with reasonable restrictions we can derive the following result.

\begin{theorem}
The convolution matrix $\frac{1}{\| \bc \|_2}W$ in \eq{eq:deconb_edge_1d} has
the RIP under the assumption that any two jumps are at least $m$ points apart.
\label{thm1}
\end{theorem}
\begin{proof}
Any sub-matrix  $W_\bt$ for $\bt$ such that the indices are at least $m$ points
apart has at most one non-zero entry in each row. In other words the non-zero
entries in one column do not overlap with non-zero entries of a different
column. Therefore  $\frac{1}{\| \bc \|^2_2} W_\bt^H W_\bt = I$, where $I$ is
the identity matrix and $ \| \frac{1}{\| \bc \|_2} W_t \bx \|^2_2 = \frac{1}{\|
\bc \|^2_2} (W_t \bx)^T (W_t \bx) =  \frac{1}{\| \bc \|^2_2} \bx^H W_t^H W_t
\bx = \| \bx \|_2^2$. And thus, the RIP holds with $\delta_S=0$.  \end{proof}

It is proven in  \cite{Candes05}  that the solution $\by$ of a convex
optimization problem of the form \eq{eq:deconb_edge_1d}, with $W$ satisfying
the RIP, satisfies \begin{equation}
        \| \by-\by^\# \|_2 \leq C_{\textrm{lp}} \| W \by^\# - \bb \|_2^2 = C_{\textrm{lp}} \|\be\|_2,
        \label{eq:cs_error}
\end{equation}
for a constant $C_{\textrm{lp}}$ that depends only on $\delta_{S}$. In other
words the error that is introduced into $\bb$, i.e. in our case due to the
proposed assumption $\bf g \approx {\bf w} * \bf u$, is at most amplified by a
constant, for an appropriate choice of the regularization parameter $\lambda$.
With $\delta_{S}=0$, $C_{\textrm{lp}} \approx 5.5$.

\begin{theorem}For a given order $m$, let $f$ be a piecewise smooth function
with $k$ continuous derivatives in intervals were $f$ is smooth, sampled on $N$
equally spaced points, with jumps at least $m$ points apart. Then,  for an
appropriate choice $\lambda$, there exists a constant $C$ that depends on $m$
and the size of the derivatives of $f$, such that the error made by solving
(\ref{eq:deconb_edge_1d}) is bound by 
\begin{eqnarray}
        \|y^\# - y\|_2 & \leq \frac{\sqrt{N}}{ N^{k_m} } C,   
\end{eqnarray}
where $k_m = \min (k,m)$. 
\label{thm2}
\end{theorem}
 
\subsection*{Note:}
Before we prove the theorem, we provide an interpretation of this result: In
particular we note that \eq{eq:deconb_edge_1d} recovers the edges exactly as $N
\to \infty$. The theorem also states that the rate of convergence to the true
jump function depends on the order of the edge detector $m$ and the
differentiability of $f$. The choice of $\lambda$ is unfortunately not trivial,
\cite{Candes05}, and depends on the underlying function $f$ and the order of
the edge detector \cite{Stefan:2012kq,SRG:08}. We also note that the provided
error bound in this theorem is a bound on the global error. In practice the
error varies locally. In particular numerical results show that the local error
is typically very small away from jumps and larger at jumps. The local nature
of the error is not captured by the stated theorem. Our numerical experiments
show also, that the local error is bounded even if jumps are closer than $m$
points. However, it is larger than the error at the jump if jumps are separated
by at least $m$ points. Overall the theorem provides an convergence guarantee,
but is not a very {\it sharp} result.  
\begin{proof}
The model error introduced by the assumption $L {\bf \Phi} \approx W \by^\#$ is given by
$$
\be = L {\bf \Phi} - W \by^\#.
$$
From the definition of $\bw$ in (\ref{eq:matching_wave_formular}) it follows
that the error is zero for piecewise constant functions. Thus, we decompose
${\bf \Phi}$ in a sum of a piecewise constant and a smooth part near the jumps. 
Let $J$ be the original domain of $f$, and let $J_1, \cdots, J_n$ be a
partition of $J$, such that each sub-interval $J_i$ contains the jump at
location $\xi_i$. We further require $\bigcup J_i = J$ and $\bigcap J_i =
\emptyset$. Let $f_1(x), \cdots, f_n(x)$ be functions defined on the whole real
line, such that $f(x) = f_i(x)$ for $x \in J_i$ and $f_i(x) \in C^k$ for all $x
\neq \xi_i$. Let 
\begin{eqnarray}
H(x) &=& \left\{ \begin{array}{cc}0 & \quad \textrm{for} \quad x < 0 \\
1 & \quad \textrm{for} \quad x \geq 0\end{array} \right.,  \nonumber \\
s_i(x) &:=& [f](\xi_i)H(x-\xi_i), \nonumber
\end{eqnarray}
and $r_i(x)=f_i(x)-s_i(x) \in C^k$. Here, $k$ is the largest number such that
$r_i$ and $f_i$ exist. Let $\bs_i$, $\br_i$ and ${\bf \Phi}_i$ be vectors with
point evaluations of $s_i(x)$, $r_i(x)$ and $f_i(x)$ respectively. Then,  we
see that 
$$
 \be_i = L {\bf \Phi}_i - W \by_i^\# = L ( \bs_i + \br_i ) - W \by_i^\# = L \br_i.
 $$
In other words the error in each subinterval can be estimated by the error that
is made by applying the edge detector to a function that is at least
continuous. It is shown in e.g. \cite{Archibaldetal05}, that the point wise
error made by the finite difference edge detector away from jumps can be bound
by: 
\begin{equation}
        \| L \br_i \|_\infty \leq \tilde C_i h^{k_m}   ,
        \label{eq:local_error}
\end{equation}
where $h$ is the grid spacing, $k_m = \min (k,m)$ and $\| \cdot \|_\infty$,
i.e. the maximum absolute value of a vector. $\tilde C_i$ is a constant that
only depends on the order $m$ of the edge detector. Let $\1$ be a the vectors
with entries $1$, then with 
$$
\| x \|_2 \leq \| \1 \|_2 \|x\|_{\infty} = \sqrt{N} \|x\|_\infty,
$$
we can bound the $l_2$ error of the edge detection by $\| \be \|_2 \leq  \tilde
C \frac{\sqrt{N}}{N^{k_m}}$. Combining this result with (\ref{eq:cs_error})
yields: 
$$
        \| y - y^\# \|_2 \leq \frac{\sqrt{N}}{ N^{k_m} } C. 
$$ 
\end{proof}





\end{document}

%% file: defs.tex
\def\bb{\mathbf{b}}
\def\bc{\mathbf{c}}

\def\bw{\mathbf{w}}

\def\by{\mathbf{y}}
\def\bx{\mathbf{x}}
\def\bu{\mathbf{u}}

\def\bc{\mathbf{c}}
\def\bs{\mathbf{s}}
\def\bt{\mathbf{t}}
\def\be{\mathbf{e}}
\def\br{\mathbf{r}}

\newcommand{\eq}[1]{(\ref{#1})}

\def\1{\mathbf{1}}

\def\mod{\textrm{mod}}

\def\xt1{\bx_i^T \1}

\def\yt1{\1^T \by^q}

\newtheorem{theorem}{Theorem}

\newenvironment{proof}[1][Proof]{\begin{trivlist}
\item[\hskip \labelsep {\bfseries #1}]}{\end{trivlist}}

%% file: for_arvix.bbl
\begin{thebibliography}{10}
\providecommand{\url}[1]{#1}
\csname url@samestyle\endcsname
\providecommand{\newblock}{\relax}
\providecommand{\bibinfo}[2]{#2}
\providecommand{\BIBentrySTDinterwordspacing}{\spaceskip=0pt\relax}
\providecommand{\BIBentryALTinterwordstretchfactor}{4}
\providecommand{\BIBentryALTinterwordspacing}{\spaceskip=\fontdimen2\font plus
\BIBentryALTinterwordstretchfactor\fontdimen3\font minus
  \fontdimen4\font\relax}
\providecommand{\BIBforeignlanguage}[2]{{%
\expandafter\ifx\csname l@#1\endcsname\relax
\typeout{** WARNING: IEEEtran.bst: No hyphenation pattern has been}%
\typeout{** loaded for the language `#1'. Using the pattern for}%
\typeout{** the default language instead.}%
\else
\language=\csname l@#1\endcsname
\fi
#2}}
\providecommand{\BIBdecl}{\relax}
\BIBdecl

\bibitem{Schweser2010}
F.~Schweser, B.~Lehr, A.~Deistung, and J.~Reichenbach, ``A novel approach for
  separation of background phase in swi phase data utilizing the harmonic
  function mean value property,'' \emph{Stockholm, Sweden: ISMRM}, p. 142,
  2010.

\bibitem{fuentesetal12a}
D.~{}Fuentes, J.~{}Yung, J.~D. {}Hazle, J.~S. {}Weinberg, and R.~J. {}Stafford,
  ``{Kalman Filtered MR Temperature Imaging for Laser Induced Thermal
  Therapies},'' \emph{Trans. Medical Imaging}, vol.~31, no.~4, pp. 984--994,
  2012.

\bibitem{Wharton2010}
\BIBentryALTinterwordspacing
S.~Wharton, A.~Sch{\"a}fer, and R.~Bowtell, ``Susceptibility mapping in the
  human brain using threshold-based k-space division,'' \emph{Magnetic
  Resonance in Medicine}, vol.~63, no.~5, pp. 1292--1304, 2010. [Online].
  Available: \url{http://onlinelibrary.wiley.com/doi/10.1002/mrm.22334/full}
\BIBentrySTDinterwordspacing

\bibitem{Rauscher:2005}
A.~Rauscher, J.~Sedlacik, M.~Barth, H.~J. Mentzel, and J.~R. Reichenbach,
  ``Magnetic susceptibility-weighted mr phase imaging of the human brain,''
  \emph{AJNR}, no.~26, pp. 736--742, 2005.

\bibitem{Duyn:2007}
J.~H. Duyn, P.~v.~Gelderen, T.~Li, J.~A. d.~Zwart, A.~P. Koretsky, and
  M.~Fukunaga, ``High-field mri of brain cortical substructure based on signal
  phase,'' \emph{PNAS}, vol.~28, no. 104, pp. 11\,796--11\,801, June 2007.

\bibitem{Haacke2004}
\BIBentryALTinterwordspacing
E.~M. Haacke, Y.~Xu, Y.-C.~N. Cheng, and J.~R. Reichenbach, ``Susceptibility
  weighted imaging (swi),'' \emph{Magnetic Resonance in Medicine}, vol.~52,
  no.~3, pp. 612--618, 2004. [Online]. Available:
  \url{http://onlinelibrary.wiley.com/doi/10.1002/mrm.20198/full}
\BIBentrySTDinterwordspacing

\bibitem{Haacke_SWI_book}
E.~M. Haacke and J.~R. Reichenbach, Eds., \emph{Susceptibility Weighted Imaging
  in MRI: Basic Concepts and Clinical Applications}.\hskip 1em plus 0.5em minus
  0.4em\relax Wiley-Blackwell, 2011, vol.~1.

\bibitem{Rochefort2008}
\BIBentryALTinterwordspacing
L.~de~Rochefort, R.~Brown, M.~R. Prince, and Y.~Wang, ``Quantitative mr
  susceptibility mapping using piece-wise constant regularized inversion of the
  magnetic field,'' \emph{Magnetic Resonance in Medicine}, vol.~60, no.~4, pp.
  1003--1009, 2008. [Online]. Available:
  \url{http://onlinelibrary.wiley.com/doi/10.1002/mrm.21710/full}
\BIBentrySTDinterwordspacing

\bibitem{Zhou2014}
D.~Zhou, T.~Liu, P.~Spincemaille, and Y.~Wang, ``{Background field removal by
  solving the Laplacian boundary value problem.}'' \emph{NMR in biomedicine},
  vol.~27, no.~3, pp. 312--9, Mar. 2014.

\bibitem{Liu2009}
\BIBentryALTinterwordspacing
T.~Liu, P.~Spincemaille, L.~de~Rochefort, B.~Kressler, and Y.~Wang,
  ``Calculation of susceptibility through multiple orientation sampling
  (cosmos): a method for conditioning the inverse problem from measured
  magnetic field map to susceptibility source image in mri,'' \emph{Magnetic
  Resonance in Medicine}, vol.~61, no.~1, pp. 196--204, 2009. [Online].
  Available: \url{http://onlinelibrary.wiley.com/doi/10.1002/mrm.21828/full}
\BIBentrySTDinterwordspacing

\bibitem{Neelavalli2009}
\BIBentryALTinterwordspacing
J.~Neelavalli, Y.-C.~N. Cheng, J.~Jiang, and E.~M. Haacke, ``Removing
  background phase variations in susceptibility-weighted imaging using a fast,
  forward-field calculation,'' \emph{Journal of Magnetic Resonance Imaging},
  vol.~29, no.~4, pp. 937--948, 2009. [Online]. Available:
  \url{http://onlinelibrary.wiley.com/doi/10.1002/jmri.21693/full}
\BIBentrySTDinterwordspacing

\bibitem{Koch2006}
\BIBentryALTinterwordspacing
K.~M. Koch, X.~Papademetris, D.~L. Rothman, and R.~A. de~Graaf, ``Rapid
  calculations of susceptibility-induced magnetostatic field perturbations for
  in vivo magnetic resonance,'' \emph{Physics in medicine and biology},
  vol.~51, no.~24, p. 6381, 2006. [Online]. Available:
  \url{http://iopscience.iop.org/0031-9155/51/24/007}
\BIBentrySTDinterwordspacing

\bibitem{Langham2009}
\BIBentryALTinterwordspacing
M.~C. Langham, J.~F. Magland, T.~F. Floyd, and F.~W. Wehrli, ``Retrospective
  correction for induced magnetic field inhomogeneity in measurements of
  large-vessel hemoglobin oxygen saturation by mr susceptometry,''
  \emph{Magnetic Resonance in Medicine}, vol.~61, no.~3, pp. 626--633, 2009.
  [Online]. Available:
  \url{http://onlinelibrary.wiley.com/doi/10.1002/mrm.21499/full}
\BIBentrySTDinterwordspacing

\bibitem{Wang2000}
Y.~Wang, Y.~Yu, D.~Li, K.~Bae, J.~Brown, W.~Lin, and E.~Haacke, ``Artery and
  vein separation using susceptibility-dependent phase in contrast-enhanced
  mra,'' \emph{Journal of Magnetic Resonance Imaging}, vol.~12, no.~5, pp.
  661--670, 2000.

\bibitem{Yao2009}
\BIBentryALTinterwordspacing
B.~Yao, T.-Q. Li, P.~v. Gelderen, K.~Shmueli, J.~A. de~Zwart, and J.~H. Duyn,
  ``Susceptibility contrast in high field mri of human brain as a function of
  tissue iron content,'' \emph{Neuroimage}, vol.~44, no.~4, pp. 1259--1266,
  2009. [Online]. Available:
  \url{http://www.sciencedirect.com/science/article/pii/S1053811908011191}
\BIBentrySTDinterwordspacing

\bibitem{Grissom2010}
W.~A. Grissom, M.~Lustig, A.~Holbrook, V.~Rieke, J.~M. Pauly, and
  K.~Butts-Pauly, ``Reweighted ℓ1 referenceless prf shift thermometry,''
  \emph{Magnetic Resonance in Medicine}, vol.~64, no.~4, pp. 1068--1077,
  October 2010.

\bibitem{Solomir2012}
R.~Salomir, M.~Viallon, A.~Kickhefel, J.~Roland, D.~R. Morel, L.~Petrusca,
  V.~Auboiroux, T.~Goget, S.~Terraz, C.~D. Becker, and P.~Gross,
  ``Reference-free prfs mr-thermometry using near-harmonic 2-d reconstruction
  of the background phase,'' \emph{IEEE Trans Med Imaging}, vol.~31, no.~2, pp.
  287--301, Feb 2012.

\bibitem{Li:2014}
\BIBentryALTinterwordspacing
W.~Li, A.~V. Avram, B.~Wu, X.~Xiao, and C.~Liu, ``Integrated laplacian-based
  phase unwrapping and background phase removal for quantitative susceptibility
  mapping,'' \emph{NMR in Biomedicine}, vol.~27, no.~2, pp. 219--227, 2014,
  nBM-13-0182.R2. [Online]. Available: \url{http://dx.doi.org/10.1002/nbm.3056}
\BIBentrySTDinterwordspacing

\bibitem{Canny:1986uq}
J.~Canny, ``A computational approach to edge detection,'' \emph{IEEE Trans.
  Pattern Anal. Mach. Intell.}, vol.~8, no.~6, pp. 679--698, 1986.

\bibitem{Mallat92}
S.~Mallat and S.~Zhong, ``Characterization of signals from multiscale edges,''
  \emph{IEEE Trans. Pattern Anal. Mach. Intell.}, vol.~14, no.~7, pp. 710--732,
  1992.

\bibitem{GelbTadmor06}
\BIBentryALTinterwordspacing
A.~Gelb and E.~Tadmor, ``{Adaptive Edge Detectors for Piecewise Smooth Data
  Based on the minmod Limiter},'' \emph{Journal of Scientific Computing},
  vol.~28, no. 2-3, pp. 279--306, Sep. 2006. [Online]. Available:
  \url{http://www.cscamm.umd.edu/people/faculty/tadmor/pub/spectral-approximations/Gelb-Tadmor.JSC2006.pdf}
\BIBentrySTDinterwordspacing

\bibitem{Lustig:2008fk}
M.~Lustig, D.~Donoho, and J.~M. Pauly, ``Sparse mri: The application of
  compressed sensing for rapid mr imaging,'' \emph{Magn Reson Med}, vol.~58,
  no.~6, pp. 1182--95, Dec 2007.

\bibitem{Stefan:2012kq}
W.~Stefan, A.~Viswanathan, A.~Gelb, and R.~Renaut, ``Sparsity enforcing edge
  detection method for blurred and noisy fourier data,'' \emph{Journal of
  Scientific Computing}, vol.~50, no.~3, pp. 536--556, 2012.

\bibitem{Archibaldetal05}
R.~Archibald, A.~Gelb, and J.~Yoon, ``{Polynomial Fitting for Edge Detection in
  Irregularly Sampled Signals and Images},'' \emph{SIAM Journal on Numerical
  Analysis}, vol.~43, no.~1, pp. 259--279, 2005.

\bibitem{SRG:08}
W.~Stefan, R.~A. Renaut, and A.~Gelb, ``Improved total variation-type
  regularization using higher-order edge detectors,'' \emph{In preparation},
  2008.

\bibitem{Wu2010}
C.~Wu and X.-C. Tai, ``Augmented lagrangian method, dual methods, and split
  bregman iteration for rof, vectorial tv, and high order models,'' \emph{SIAM
  J. Imaging Sciences}, vol.~3, no.~3, pp. 300--339, 2010.

\bibitem{Taylor:2008}
\BIBentryALTinterwordspacing
B.~A. Taylor, K.~P. Hwang, A.~M. Elliott, A.~Shetty, J.~D. Hazle, and R.~J.
  Stafford, ``Dynamic chemical shift imaging for image-guided thermal therapy:
  analysis of feasibility and potential,'' \emph{Med Phys}, vol.~35, no.~2, pp.
  793--803, Feb 2008. [Online]. Available:
  \url{http://www.hubmed.org/display.cgi?uids=18383702}
\BIBentrySTDinterwordspacing

\bibitem{saxenaetal2007}
R.~Saxena, A.~Gelb, and H.~D. Mittelmann, ``{A High Order Method for
  Determining the Edges in the Gradient of a Function},'' \emph{Communications
  in Computational Physics}, vol.~5, no. 2-4, pp. 694--711, 2009.

\bibitem{Vigen:2003aa}
K.~K. Vigen, B.~L. Daniel, J.~M. Pauly, and K.~Butts, ``Triggered, navigated,
  multi-baseline method for proton resonance frequency temperature mapping with
  respiratory motion,'' \emph{Magn Reson Med}, vol.~50, no.~5, pp. 1003--10,
  Nov 2003.

\bibitem{Candes05}
E.~J. Cand{\`e}s and T.~Tao, ``Near-optimal signal recovery from random
  projections: Universal encoding strategies?'' \emph{IEEE Transactions on
  Information Theory}, vol.~52, no.~12, pp. 5406--5425, 2006.

\end{thebibliography}
